\documentclass[12p, reqno]{amsart}
\usepackage{amsfonts}
\usepackage{amsmath,amsthm,graphicx,mathrsfs,url   }
\usepackage{amssymb}
\usepackage{enumitem}

\usepackage{graphicx}
\usepackage[usenames,dvipsnames]{color}
\usepackage[colorlinks=true, linkcolor=Red, citecolor=Green]{hyperref}
\usepackage[displaymath,mathlines]{lineno}
\def\squarebox#1{\hbox to #1{\hfill\vbox to #1{\vfill}}}

\newcommand{\R}{{\mathbb R}}
\newcommand{\C}{{\mathbb C}}
\newcommand{\N}{{\mathbb N}}

\renewcommand{\Re}{\mathop{\rm Re}\nolimits}
\renewcommand{\Im}{\mathop{\rm Im}\nolimits}

\theoremstyle{plain}

\newtheorem{thm}{Theorem}[section]

\newtheorem{lem}{Lemma}[section]
\newtheorem{rem}{Remark}
\newtheorem{prop}{Proposition}[section]

\numberwithin{equation}{section}

\def\re{\mbox{\rm Re}\:}

\def\R{{\mathbb R}}
\def\C{{\mathbb C}}

\def\hc{{\mathcal H}}
\def\pa{\partial}

\def\ii{{\bf i}}

\def\cc{{\mathcal C}}

\def\f3{\frac{3}{2}\delta}
\def\ep{\epsilon}

\def\supp{{\rm supp}\:}

\def\12{\frac{1}{2}}

\def\phi{\varphi}
\def\epsilon{\varepsilon}
\def\kappa{\varkappa}
\def\ep{\epsilon}

\def\dc{{\mathcal D}}
\def\oc{{\mathcal O}}
\def\pa{\partial}

\def\id{{\rm Id}}
\def\g{\Gamma}

\begin{document}

\title[Absence of eigenvalues]{Absence of eigenvalues of dissipative operator for strictly convex obstacles}\author[V. Petkov]{Vesselin Petkov}

\address{Universit\'e de Bordeaux, Institut de Mathématiques de Bordeaux, 351, Cours de la Libération,
33405 Talence, France}
\email{petkov@math.u-bordeaux.fr} 

\maketitle
\begin{abstract}
We study the wave equation in the exterior of a strictly convex bounded domain $K \subset \R^d, d \geq 3,$ odd, with dissipative boundary condition $\pa_{\nu} u - \gamma(x) \pa_t u = 0$ on the boundary $\Gamma$ and $0 < \gamma(x) <1, \:\forall x \in \Gamma.$ The solutions are described by a contraction semigroup $V(t) = e^{tG}, \: t \geq 0.$  In \cite{P3} we established that for $\gamma \equiv const$  and $K = \{x \in \R^3: \:|x| \leq 1\}$ the operator $G$ has no eigenvalues and we conjectured that the same result holds for every strictly convex obstacle.  In this paper we prove this conjecture.
 \end{abstract}

\section{Introduction}
Let $K \subset \R^d,$ $d \geq 3$, $d$ odd, be a bounded non-empty domain with $C^{\infty}$ strictly convex boundary $\Gamma.$   Let $\Omega = \R^d \setminus \bar{K}$ be connected and $K \subset \{x \in \R^d:\:|x| < \rho\}.$ 
Consider the boundary problem
\begin{equation} \label{eq:1.1}
\begin{cases} u_{tt} - \Delta_x u = 0 \: {\rm in}\: \R_t^+ \times \Omega,\\
\partial_{\nu}u - \gamma(x) \pa_t u= 0 \: {\rm on} \: \R_t^+ \times \Gamma,\\
u(0, x) = f_1, \: u_t(0, x) = f_2 \end{cases}
\end{equation}
with initial data $(f_1, f_2) \in \hc = H^1(\Omega) \times L^2(\Omega).$
Here $\nu(x)$ is the unit outward normal at $x \in \Gamma$ pointing into $\Omega$ and $\gamma(x) > 0$ is a $C^{\infty}$ function on $\Gamma.$ 
Introduce the contraction  semi-group $V(t)$  in ${\mathcal H}$ with  generator
$$ G = \Bigl(\begin{matrix} 0 & 1\\ \Delta & 0 \end{matrix} \Bigr).$$
The operator $G$  has domain  $\dc$ given by the closure in the graph norm
$$|\|f\| | = (\|f\|_{{\mathcal H}}^2 + \|G f\|^2_{{\mathcal H}})^{1/2} $$
 of functions $f = (f_1, f_2) \in C_{(0)}^{\infty} (\R^d) \times C_{(0)}^{\infty} (\R^d)$ satisfying the boundary condition $\partial_{\nu} f_1 - \gamma f_2 = 0$ on $\Gamma.$

The solution of the problem (\ref{eq:1.1}) has the form $V(t)f = e^{tG} f,\: t \geq 0$. The point spectrum $\sigma_p(G)$ of $G$ in $\C_{-} = \{z \in \C:\:\Re z < 0\}$ is formed by isolated eigenvalues with finite multiplicity, $\sigma_p(G) \cap \ii \R = \emptyset$ and $G$ has continuous spectrum $\sigma_c(G) = \ii \R$ (see Section 8, \cite{LP1}).
Notice that if $Gf =\lambda f$ with $0 \neq f \in \mathcal D, \: \Re \lambda < 0,$ we have 
 solution $u(t, x) = V(t) f = e^{\lambda t} f(x) $  of (\ref{eq:1.1}) with exponentially decreasing global energy. Such solutions are called asymptotically disappearing and they are important for scattering problems (see for more details \cite{P1}, \cite{P3}). The location in $\C_{-}$ of the eigenvalues of $G$ has been studied in \cite{P1} improving previous results of Majda \cite{Ma}. In the special case when $\gamma \equiv 1,$  and $K$ is the unit ball $B_3 = \{x \in \R^3: \;|x|\leq 1\}$ it was proved in \cite{P1} that $G$ has no eigenvalues. 
The case $ \gamma(x) > 1,\: \forall x \in \g$ has been examined in \cite{P1} and a Weyl formula for the counting function of the eigenvalues of $G$ close to the negative real axis has been established in \cite{P2}.

   Throughout this paper we assume that $0 < \gamma(x) < 1, \: \forall x \in \g$.  This case has been studied in \cite{P1}, \cite{P3}. In particular, for $\gamma \equiv const$ and $K = B_3$  in Appendix, \cite{P3} it was proved that $G$ has no eigenvalues. We conjectured in \cite{P3} that the same is true for every strictly convex obstacle. Our main result  is the following
   \begin{thm} For strictly convex obstacles the operator $G$ has no eigenvalues $\lambda$ with $\Re \lambda < 0.$
   \end{thm}
  
  The eigenvalues of $G$ in $\C_{-}$ are incoming functions (see Section 2 for the definition of outgoing/incoming functions). The absence of eigenvalues of $G$  for $\Re\lambda > 0$ is based on the representation of the outgoing resolvent
  \begin{equation} \label{eq:1.2}
 (G - \lambda)^{-1} = -\int_0^{\infty} e^{-\lambda t} e^{t G} dt,\:\Re \lambda > 0
 \end{equation} 
which is analytic for $\Re\lambda > 0$. For the incoming eigenvalues it should be useful to have a similar representation with integral on $(-\infty, 0]$ involving some evolution operator. However, the problem
\begin{equation} \label{eq:1.3}
\begin{cases} (\pa^2_t - \Delta_x) f = 0 \: {\rm in}\: \R_t^- \times \Omega,\\
\partial_{\nu}f - \gamma(x) \pa_t f= 0 \: {\rm on} \: \R_t^- \times \Gamma,\\
f(0, x) = 0, \: f_t(0, x) = \chi(x) \end{cases}
\end{equation}  
is not $L^2$  well posed (see Theorem 1 in \cite{Ma}).
 On the other hand, the problem (\ref{eq:1.3}) is well posed with some loss of regularity (see \cite{I2}, \cite{I3}) and  for $\chi \in H_{comp}^m(\Omega), \: m > d/2 + 3$ it is possible to obtain a solution of (\ref{eq:1.3}) with the property
 $e^{ c_{m- 3/2} t} f(t, x) \in H^{m- 2}(\R_t^{-} \times \Omega)$ with some large $c_{m  - 3/2} > 0$ (see Proposition 3.2). The idea is to consider the incoming function
 \[ w(x, \lambda) = \int_{-\infty}^0 e^{- \lambda t} P(t) \chi dt, \: \Re \lambda \leq - c_{m- 3/2},\]
 where $P(t) \chi$ is a solution of (\ref{eq:1.3}) and to modify it to obtain an approximative  resolvent $R_a^{\sharp}(\lambda)$ of $(G - \lambda)^{-1}.$ 
 
  We expose  below the strategy of the proof of Theorem 1.1. In Section 3 we examine the operator $\cc(\lambda) = N(\lambda) - \lambda \gamma(x),$ where $N(\lambda): H^s(\g) \to H^{s-1}(\g)$ is the Dirichlet-to-Neumann operator defined in Section 2. We study the existence of $\cc(\lambda)^{-1}$ for $\Re \lambda \leq - c_0 < 0.$ For this purpose we apply a weak version of Theorem 2.1 in \cite{I2} concerning only $N(\lambda)$. We expect that it is possible to obtain the existence of $\cc(\lambda)^{-1}$  for large negative $\Re \lambda$ for general obstacles and we will study this problem in a further work. In Proposition 3.2 we justify the existence of an operator $P(t)$ satisfying (\ref{eq:3.13}). In Section 4 we follow the approach of the proof of Theorem 4.43 in \cite{DZ} by using a cut-off function $\zeta_a(t, x)$. Let $\chi_a \in H_{comp}^{m}(\Omega),\: \chi_a(x) \equiv 1$ for $|x| \leq \rho + 2 \ep,\: \ep > 0.$ We study the operator
 $\zeta_a P(t) \chi_a$ which satisfies
 \[\begin{cases}
 (\pa_t^2 - \Delta_x) (\zeta_a P(t) \chi_a) = : F_a(t),\\
( \pa_{\nu} - \gamma(x) \pa_t) (\zeta_a P(t) \chi_a)\vert_{\R_t^{-} \times \g} = : J_a(t). 
 \end{cases} \]
 To treat the remainder operators $F_a(t)$ and $J_a(t)$, we define an approximative resolvent
  \[R^{\sharp}_a(\lambda) = \int_{-\infty}^0 e^{-\lambda t} \Bigl[\zeta_a P(t)\chi_a  + (1- \chi_c) W_a(t)- B_a(t) \Bigr]dt, \: \Re \lambda \leq 0.\]
  with cut-off function $ \chi_c(x)$ and suitable correction operators $W_a(t),\: B_a(t)$. The kernels of these operators have compact support with respect to $t$ and this makes possible to extend analytically $R^{\sharp}_a(\lambda)$ from $\Re \lambda \leq - c_{m- 3/2}$ to $\Re\lambda < 0.$
  The operator $R^{\sharp}_a(\lambda)$ satisfies $(\Delta - \lambda^2) R^{\sharp}_a(\lambda) = \chi_a(\id + L(\lambda))$ and 
  for sufficiently large $|\lambda| \geq A_0$ we obtain the existence of $(\id + L(\lambda))^{-1}.$ This leads to the absence of eigenvalues $\lambda$ of $G$ with $|\lambda| \geq A_0.$
  Finally,  Section 5 is devoted to the absence of eigenvalues with $|\lambda| \leq A_0.$ For $\ep \in [0, 1]$ we introduce the operator $\cc_{\ep}(\lambda) = N(\lambda) - \lambda \ep\gamma(x)$ and the generator $G_{\ep}$ related to problem (\ref{eq:1.1}) with boundary condition $(\pa_{\nu}u - \ep \gamma \pa_t u)\vert_{\g} = 0.$ We introduce the bounded set $\omega =\{ z \in \C:\: - c_0 \leq \Re\lambda \leq 0,\: |\Im z| \leq A_0\}$ and define $\eta = \sup\{\ep: \:\ep \in [0,1],\: G_{\ep} \:\text{ has no eigenvalues in} \: \omega\}.$  The problem is reduced to the analysis of the case when  $G_{\eta}$ has an eigenvalue $\lambda_0 \in \omega.$  By a perturbation argument for $G_{\ep}$ and $\eta- \ep > 0$ sufficiently small we show that this is impossible.
  
  \section{Preliminaries}

In this section we collect some facts from \cite{P1}, \cite{P3}.  For $\lambda \in \C$ introduce the exterior Dirichlet-to-Neumann map
$$N(\lambda): H^s(\Gamma) \ni f \longrightarrow \pa_{\nu} u\vert_{\Gamma} \in H^{s-1}(\Gamma),$$
where $u = K(\lambda) f$ is the solution of the problem
\begin{equation} \label{eq:2.1}
\begin{cases} (-\Delta +\lambda^2) K(\lambda)f = 0 \: {\rm in}\: \Omega,\\
K(\lambda)f = f \:{\rm on}\: \Gamma,\\
K(\lambda) f : \lambda-{\rm incoming}.
 \end{cases}
\end{equation}

A function $u(x)$  is called $\lambda$-{\it outgoing} ($\lambda$-{\it incoming}) if there exists $R > \rho$ and $g \in L^2_{comp}(\R^d)$ such that
$$u(x) = R_0^{\pm}(\lambda) g,\: |x| \geq R,$$
where $R_0^{\pm}(\lambda) =(-\Delta_0 +\lambda^2)^{-1}$ is the outgoing (+) (incoming (-)) resolvent of the free Laplacian $- \Delta_0$ in $\R^d$. The resolvents $R_0^{\pm}(\lambda)$  are analytic in $\C$ and they have kernels
 \begin{equation*}   R_0^{\pm} (\lambda, x, y) = \frac{(-1)^{(d-1)/2}}{2(2 \pi)^{(d-1)/2}} \Bigl(\frac{1}{r}\pa_r\Bigr)^{(d- 3)/2} \Bigl(\frac{ e^{\mp \lambda r}}{r}\Bigr)\Big\vert_{r =  |x-y|}.
\end{equation*}

 If for $\Re \lambda < 0$ we have  $G(f_1, f_2) = \lambda  (f_1, f_2)$, then $f_1 \in H^2(\Omega)$ is $\lambda$-incoming solution of $(-\Delta + \lambda^2) f_1 = 0$ (see \cite{LP1} and Section 1 in \cite{P3}). Let $R_D (\lambda) = (-\Delta_D +\lambda^2)^{-1}$ be the incoming resolvent of the Dirichlet Laplacian $\Delta_D$ in $\Omega$ with domain ${\bf D} = H^2(\Omega) \cap H^1_0(\Omega)$
 which is analytic for $\lambda \in \C_{-} = \{ z \in \C: \: \Re \lambda \leq 0\}$. Let
 $${\bf D}_{loc} = \{ u \in L^2_{loc}(\Omega):\: \chi(x) \in C_0^{\infty}(\R^d), \: \chi(x) \equiv 1\: {\text in\:a\:neighborhood\:of} \bar{K}\: \Rightarrow \chi u \in {\bf D} \}. $$
 The incoming resolvent has meromorphic extension 
 $$R_D (\lambda):\: L^2_{comp}(\Omega) \rightarrow {\bf D}_{loc}$$
  from $\C_{-}$ to $\C$ . The solution of the problem
(\ref{eq:2.1}) with $f \in H^{3/2}(\Gamma)$ has the form
\begin{equation} \label{eq:2.2}
u = e(f) + R_D(\lambda)((\Delta -\lambda^2) (e(f)),
\end{equation} 
where $e(f): H^{3/2}(\Gamma) \ni f \to e(f) \in H^{2}_{comp}(\Omega)$ is an extension operator. Clearly, we may find $\pa_{\nu}u\vert_{\Gamma}$ by applying  (\ref{eq:2.2}).
 By (\ref{eq:2.2}) we conclude that $N(\lambda)$ is analytic in $\bar{\C}_{-}.$

We write the boundary condition  in (\ref{eq:1.1}) as 
\begin{equation*} 
\cc(\lambda)v: = (N(\lambda) - \lambda \gamma) v  = \Bigl( \id - \lambda \gamma N(\lambda)^{-1} \Bigr) N(\lambda)v  = 0,\:v=  f_1\vert_{\Gamma} \in H^{3/2}(\Gamma).
\end{equation*} 
For $\Re  \lambda \leq 0$ the operator $\cc(\lambda):\:H^{s}(\Gamma) \rightarrow H^{s-1}(\Gamma)$ has the same singularities as $N(\lambda)$.
On the other hand, $N(\lambda)^{-1}:\: H^{s}(\Gamma) \rightarrow H^{s+ 1}(\Gamma)$ is compact operator and by analytic Fredholm theorem (see \cite{P1}, \cite{P3}),  the operator $\cc(\lambda)$ is a meromorphic operator valued function for $\Re \lambda \leq 0.$ Notice that if $\cc(\lambda)$ is analytic at $\lambda_0$, then $\cc(\lambda_0): \: H^{s}(\Gamma) \rightarrow H^{s+1}(\Gamma).$ Here and below a meromorphic operator valued function $B(z)$ means that $B(z)$  have Laurent expansion with finite number negative powers of $z$ and coefficients having finite rank. In \cite{P2} it was proved that we can extend  the incoming   resolvent $(G - \lambda)^{-1},\: \lambda \notin \sigma_p(G), \: \re \lambda < 0$ as meromorphic function $(G - \lambda)^{-1}:\: {\mathcal H}_{comp} \rightarrow {\mathcal D}_{loc}$ for $\lambda \in \C,$  
  where
  $${\mathcal D}_{loc} = \{ u \in {\mathcal H}_{loc},\: \chi(x) \in C_0^{\infty}(\R^d), \: \chi(x) \equiv 1 \: {\text in\: a\: neighborhood\:of} \bar{K}\: \Rightarrow \chi u \in {\mathcal D} \}.$$ 
 Let $ \begin{pmatrix} u \\ w \end{pmatrix} = (G- \lambda)^{-1} \begin{pmatrix} f \\ g \end{pmatrix} $  with $(f, g) \in {\mathcal H}_{comp}$. It was established in \cite{P1},  \cite{P3} that $(G - \lambda)^{-1}$ is analytic at $\lambda_0$ if and only if $\cc(\lambda_0)$ is invertible and $\cc(\lambda_0)^{-1} : H^s(\Gamma) \rightarrow H^{s+1}(\Gamma) $ is analytic at $\lambda_0.$ Moreover, if $\cc(\lambda)^{-1}$ is analytic, 
  $w = \lambda u + f,$ where 
\begin{equation} \label{eq:2.3}
 u(x, \lambda) = -  R_D(\lambda)   (g + \lambda f) +  K(\lambda) \cc(\lambda)^{-1} \Bigl[ \pa_{\nu}\Bigl (R_D(\lambda) ( g + \lambda f)\Bigr)\Big\vert_{\Gamma}  + \gamma f\vert_{\Gamma} \Bigr]
\end{equation}
is a solution of the problem
\begin{equation} \label{eq:2.4}
\begin{cases}
(\Delta - \lambda^2) u = g + \lambda f \: {\rm in}\: \Omega,\\
(\pa_{\nu} - \gamma(x) \lambda) u \vert_{\Gamma} = \gamma f\vert_{\Gamma},\\
u: \lambda-{\rm incoming}. \end{cases}
\end{equation}
Notice that for $(f, g) \in {\mathcal H}_{comp}$ we obtain $u(x, \lambda) \in H^2(\Omega).$ Indeed,
\[\cc(\lambda)^{-1} \Bigl[\pa_{\nu}R_D(\lambda) ( g + \lambda f)\Big\vert_{\g}  + \gamma f\vert_{\Gamma} \Bigr] \in H^{3/2}(\Gamma)\]
and applying $K(\lambda)$, we obtain the result. 

It is easy to see that if $\cc(\lambda)^{-1}: \: H^s(\g) \rightarrow H^{s+ 1}(\g)$ for some $s_0 \geq 0$ has a pole $\lambda \in \bar{\C}_{-},$ then this operator has a pole $\lambda_0$ for every $s \geq 0.$ This follows from the fact that there exists $0 \neq v \in H^s(\g)$ such that $\cc(\lambda_0)v = (\id - \lambda_0 \gamma N(\lambda_0)^{-1}) v = 0.$ Since $N(\lambda_0)^{-1}: H^s(\g) \to H^{s+1}(\g)$ is bounded, this implies $v \in C^{\infty}(\g)$ and $(\id - \lambda_0 \gamma N(\lambda_0)^{-1})$ is not invertible in $H^s(\g)$ for all $s \geq 0.$ In Section 4 we need the following

\begin{lem} Assume that there exist $a_1 > \rho$ and $k \geq 2$ such that for $f = 0$ and every $g \in H^k(\Omega)$ with $\supp\: g \subset B(0, a_1)$ there exists a solution $u(x, \lambda)$ of $(\ref{eq:2.4})$ which for $x \in \Omega \cap\{x \in \R^d:\:|x| \leq a, \: a > \rho\}$ is analytic in a neighbourhood $\{\lambda \in \C:\: |\lambda -\lambda_0| < \delta \}$ of $\lambda_0 \in \C_{-}.$ Then  the operator $\cc(\lambda)^{-1}: H^{k +1}(\g) \rightarrow H^{k+2}(\g)$ is analytic at $\lambda_0$. 
\end{lem} 
\begin{proof}Suppose that  $\cc(\lambda)^{-1}$ has a pole at $\lambda_0.$ Then there exist $h \in H^{k+ 1}(\g),\: 0 \neq h_0 \in H^{k+2}(\g)$ and $m \geq 1$ such that for $0 \neq |\lambda - \lambda_0| < \ep$ with sufficiently small $0 < \ep \leq \delta$ we have
\[\cc(\lambda)^{-1} h=\frac{h_0}{(\lambda - \lambda_0)^m} + \text {lower order singularities} .\]
Let $q = N(\lambda_0)^{-1}h \in H^{k + 2}(\Gamma).$ It is easy to construct $\tilde{g} \in H^{k + 2}(\Omega)$ with $\supp \:\tilde{g} \subset B(0, a_1)$ such that $\tilde{g}\vert_{\Gamma} = q, \: \pa_{\nu}\tilde{g}\vert_{\Gamma} = 0.$ To do this, by using a partition of unity on $\Gamma,$ we pass to a construction in a neighbourhood $\mathcal U \subset \R^d$ of a point 
$x_0 \in \Gamma.$ Consider in ${\mathcal U}$ local normal coordinates $(y_1,...,y_d) = (y_1, y')$ such that
\[ x = \alpha(y_1, y') = \beta(y') + y_1 \nu(y'), \]
where $ y_1 = {\rm dist}\: (x, \Gamma)$ and $\nu(y')$ is an extension of unit normal vector to unit vector field. Then the boundary $\Gamma$ is given by $y_1 = 0$ and 
$\pa_{\nu} v = \pa_{y_1}v (\alpha(y_1, y)).$ If $ q(\alpha(0, y')) = s(y')$, we take $\tilde{g}(y) = s(y').$ 
Next for $f = 0$ and $g = (\Delta- \lambda_0^2)\tilde{g}\in H^{k}(\Omega)$ by our assumption there exists a solution $u(x, \lambda)$ of (\ref{eq:2.4}) which is analytic for  $x \in \Omega,\: |x| \leq a,\:|\lambda -\lambda_0| < \delta$. The function $v = R_D(\lambda_0) g$ satisfies
\[  (-\Delta + \lambda_0^2)(v + \tilde{g}) = 0,\: (v + \tilde{g})\vert_{\Gamma} = q,\]
hence $\pa_{\nu} v\vert_{\Gamma} = \pa_{\nu} (v+ \tilde{g})\vert_{\Gamma} = N(\lambda_0)q = h.$ Then in  the representation (\ref{eq:2.3}) for  $u(x, \lambda)\vert_{\g}$ we will have a singular term given  by $\cc(\lambda)^{-1} \pa_{\nu} R_D(\lambda_0)g)\vert_{\g}$ and we obtain a contradiction with the analyticity of $u(x, \lambda)$ for $|x| \leq a.$
\end{proof}
\begin{rem} By the same argument we may construct a modification $\tilde{g} \in H^m(\Omega)$ of $g \in H^m(\Omega)$ such that $\tilde{g}\vert_{\g} = 0,\: \pa_{\nu}\tilde{g}\vert_{\g} = g\vert_{\g}$. To do this, in local normal coordinates $(y_1, y')$ we take $\tilde{g}(y) = y_1 g(\alpha(y_1, y')).$
\end{rem}

\section{Existence and estimate of $\cc(\lambda)^{-1}$}
\subsection{Analyticity of $\cc(\lambda)^{-1}$}

Throughout the paper we will use the notation $\lambda = \mu + \ii k,\: \mu, \: k \in \R.$ In this subsection we study the analyticity of $\cc(\lambda)^{-1}$ for sufficiently negative $ \mu$. Denote by $( . , .)_s,\: \| . \|_s$ are the scalar product and the norm  in $H^s(\g),$ respectively. Consider the problem
 \begin{equation*} 
\begin{cases} (-\Delta +\lambda^2) K_1(\lambda)f = 0 \: {\rm in}\: \Omega,\\
K_1(\lambda)f = f\:{\rm on}\: \Gamma,\\
K_1(\lambda) f: \lambda-{\rm outgoing}.
 \end{cases}
\end{equation*} 
 and the Dirichlet-to-Neumann operator  $N_1(\lambda) f = \pa_{\nu} K_1 f(\lambda)\vert_{\g}.$ 
  Ikawa established the following
 \begin{thm} [Theorem 2.1 in \cite{I2}] Let $K$ be a strictly convex obstacle in $\R^3$. Then there exist constants $r > 0, C_m > 0$ such that for every $m \in \N$ and $\mu \geq r$ we have
 \begin{equation} \label{eq:3.1}
 - \Re ( N_1(\lambda) \varphi, \varphi)_m \geq (\mu - C_{m})\|\varphi\|_m^2,\:\: \forall \varphi \in C^{\infty}(\g)
 \end{equation}
  with $C_m$ independent of $\lambda$ and $\varphi.$ 
\end{thm}
After the change $\eta = - \lambda,$ we deduce for operator $N(\lambda)$ related to  (\ref{eq:2.1}) and $\mu \leq - r$ the estimate
\begin{equation}\label{eq:3.2}
 - \Re ( N(\lambda)  \varphi, \varphi)_m \geq ( |\mu|  - C_{m})\|\varphi\|_m^2,\:\: \forall \varphi \in C^{\infty}(\g).
 \end{equation}
 \begin{rem}
 We apply the estimate $(\ref{eq:3.1})$ in \cite{I2} with $\gamma \equiv 0$. This estimate can be proved for odd dimensions $d > 3$ by similar argument. Notice that some estimates of the operator  $\frac{N(\lambda)}{\lambda}$ in suitable regions have been established in Section $4$ in \cite{P3}. 
 \end{rem}
Let $\delta = \max_{x \in \g} \gamma(x) < 1.$ For $m = 0$ we have $(\gamma \varphi, \varphi)_0 \leq \delta\|\varphi\|_0^2$ and we obtain the estimate
 \[ - \Re ((N(\lambda) - \lambda \gamma(x))\varphi, \varphi) _0 \geq \Bigl((1 -\delta)|\mu|- C_0\Bigr) \|\varphi\|_0^2.\]
 Setting $\alpha = 1 - \delta> 0,\: c_0 = \max \{ \frac{2C_0}{\alpha}, r\},$ for $\mu \leq -c_0$ we may absorb the term $C_0\|\varphi\|_0^2$ and obtain
\begin{equation} \label{eq:3.3}
\| \cc(\lambda) \varphi\|_0 \geq \frac{\alpha |\mu|}{2} \|\varphi\|_0, \: \: \forall \varphi \in C^{\infty}(\g).
\end{equation}
In the following we suppose that $\mu \leq - c_0.$ Set $M(\lambda) = \Bigl( Id - \lambda  \gamma N^{-1}(\lambda)\Bigr)$ and note that $\cc(\lambda) =  M(\lambda) N(\lambda).$ If for $s \geq 0$ the operator $\cc(\lambda)^{-1}: H^s (\Gamma) \rightarrow H^{s+1}(\Gamma)$ has a pole at $\lambda_0$, then $M(\lambda_0): H^s(\Gamma) \to H^s(\Gamma)$ is not invertible and there exists $0 \neq  \varphi_0 \in L^2(\Gamma)$ such that $\cc(\lambda_0) \varphi_0 =  M(\lambda_0) \varphi_0 = 0.$ Since $N(\lambda_0)^{-1} : H^s(\Gamma)  \rightarrow H^{s+ 1}(\Gamma),$ we conclude that $\varphi_0 \in C^{\infty}(\Gamma).$ This leads to a contradiction with (\ref{eq:3.3}), hence $\cc(\lambda)^{-1}$ is analytic for $\mu \leq - c_0$ and $c_0$ is independent of $s.$
\begin{rem} The invertibility of $\cc(\lambda)$ for $s \in \N$ and $\mu \geq - c_0$ has been proved in \cite{I2} following  Lemma $3.3$ in \cite{I1}. Our argument above is short and we use $(\ref{eq:3.2})$ only for $m = 0.$
\end{rem} 
 Now we will estimate the norm of $\cc(\lambda)^{-1}: H^s(\g) \to H^s(\g).$ 
 Let $h \in L^2(\g)$ and let $C(\lambda) g = h $ with $g \in H^{1}(\g)$. (We omit in the notations the dependence of $g$ of $\lambda$). Choose a sequence $\varphi_j \rightarrow_{j \to \infty} g$ in $H^{1}(\g)$ with $\varphi_j \in C^{\infty}(\g).$ Therefore
$$\|C(\lambda) (\varphi_j - g)\|_0 \leq B_0(\lambda) \|\varphi_j - g\|_{1}$$
and $C(\lambda) \varphi_j \to C(\lambda)g$ in $L^2(\g)$. On the other hand, $\|\varphi_j\|_0 \to \|g\|_0.$ We apply (\ref{eq:3.3}) with $\varphi_j$ and passing to limit, we deduce
\begin{equation} \label{eq:3.4}
\| \cc(\lambda)^{-1} h\|_0 \leq \frac{2}{\alpha |\mu|}\|h\|_0,\: \forall h \in L^2(\g).
\end{equation}
Next assume that for $\mu \leq -c_m \leq - c_0$ we have the estimate
\begin{equation} \label{eq:3.5}
\| \cc(\lambda)^{-1} h\|_m \leq \frac{B_m}{\alpha |\mu|}\|h\|_m,\: \forall h \in H^m(\g).
\end{equation}
Applying (\ref{eq:3.2}), we get
\[ - \Re ((N(\lambda) - \lambda \gamma(x))\varphi, \varphi)_{m+1} \geq (|\mu| - C_{m+ 1}) \|\varphi\|_{m+1}^2 + \mu (\gamma \varphi, \varphi)_{m + 1}.\]
On the other hand,
\[ |\mu (\gamma \varphi, \varphi)_{m+1}| \leq |\mu|  \delta \|\varphi\|_{m+1}^2 + A_{m} |\mu| \|\varphi\|_m \|\varphi\|_{m+1},\] 
where $A_{m}$ depends of the norms of the derivatives of $\gamma(x)$ of order $j \leq m.$
The above estimates imply
\[ A_{m} |\mu| \|\varphi\|_{m} \|\varphi\|_{m+1} +\|\cc(\lambda) \varphi\|_{m+1} \| \varphi\|_{m+1} \geq (\alpha |\mu| - C_{m+1}) \|\varphi\|^2_{m+1}.\]
We estimate $|\mu| \|\varphi\|_m$ from (\ref{eq:3.5}) and obtain
\[ (A_m B_m \alpha^{-1} + 1)\|\cc(\lambda) \varphi\|_{m+1} \geq (\alpha |\mu| - C_{m+1})\|\varphi \|_{m+1}. \]
Choosing $c_{m+1} = \max\{ \frac{2 C_{m+1}}{\alpha}, c_m\}$  and repeating the above argument with an approximation $\varphi_j \to g$ in $H^{m+2}(\g)$,  we conclude for $\mu \leq - c_{m+1}$ that
\[ \|\cc(\lambda)^{-1} h\|_{m+1} \leq \frac{B_{m+1}}{\alpha |\mu|} \|h\|_{m+1},\: \forall h \in H^{m+1}(\g).\]
By iteration we establish this estimate for every $m \in \N$. Next applying an interpolation argument for the spaces $H^m(\Gamma)$ and $H^{m+ 1}(\Gamma),$ we extend it  to $H^s(\Gamma), \: m < s < m+ 1$ and obtain the following
\begin{prop} For every $ s \geq 0$ there exist constants $c_s > 0, \: B_s > 0$ such that $\cc(\lambda)^{-1}: H^s(\g) \rightarrow H^{s+1}(\g)$  is analytic for $\mu \leq - c_s$ and in this domain we have
\begin{equation} \label{eq:3.6}
\|\cc(\lambda)^{-1} h\|_s \leq \frac{B_s}{\alpha |\mu|} \|h\|_s,\: \forall h \in H^s(\g).
\end{equation}
\end{prop}

\subsection{Existence of solution of the problem (\ref{eq:1.3})}
Let $\chi \in H^m(\Omega)$ with $m \in \N,\: m > d/2 + 3$ and $ \supp \chi \subset B(0, a) =\{x: \: |x| \leq a\}, \: a > \rho.$ 
 Consider an extension $\tilde{\chi} \in H^m(\R^d)$ such that $\tilde{\chi}\vert_{\Omega} = \chi$ and $\|\tilde{\chi}\|_{H^m(\R^d)}\leq M \|\chi\|_{H^m(\Omega)}.$ Let $g_0(t, x)$ be the solution of 
 \begin{equation} \label{eq:3.7}
 \begin{cases} (\pa_t^2-\Delta_x) g_0= 0\: {\rm in}\: \R \times \Omega,\\
g_0(0, x)  =  0, \: \pa_t g_0(0, x) = \tilde{\chi}(x).
 \end{cases}
 \end{equation}
 The function $g_0$ has the representation $g_0 = \frac{\sin(t\sqrt{-\Delta_0})}{\sqrt{-\Delta_0}}\tilde{\chi},$ where $-\Delta_0$ is the free Laplacian in $\Omega.$
 Set
 $$w_0(\lambda, x) = \int_{-\infty}^{\infty} e^{-\lambda t} H(t)g_0(t, x) dt, \: \Re \lambda \leq 0,$$
 where throughout our paper we use the Heaviside function $H(t) = 1$ for $ t \leq 0$ and $H(t) = 0$ for $t > 0.$
 We obtain   $(-\Delta + \lambda^2) w_0 = \tilde{\chi}$ and
 \begin{equation} \label{eq:3.8}
 q(\lambda, x) = (\pa_{\nu} - \lambda \gamma(x)) w_0(\lambda, x)\vert_{\Gamma} = \int_{-\infty}^{\infty} e^{-\lambda t} H(t)(\pa_{\nu} - \gamma \pa_t)g_0(t, x) \vert_{\Gamma}dt.
 \end{equation} 
 Thus  $w_0(\lambda, x) = R_0(\lambda)\tilde{\chi} \in H^{m+2}(\R^d),$ where $R_0(\lambda) = (- \Delta_0 + \lambda^2)^{-1}$ is the incoming resolvent of $-\Delta_0$ 
 and $w_0(\lambda, x)$ is analytic for $\Re \lambda \leq 0$.  Let $\| . \|_{H^m(B(0, b))} = \| .\|_{m, b}, \: b > \rho.$ It is well known that
 \begin{equation} \label{eq:3.9}
 \|R_0(\lambda)\chi\|_{k, b} \leq C_{a, b} |\lambda|^{k-1}\|\tilde{\chi}\|_{0, a},\:\:\Re\lambda \leq 0,  |\lambda| \geq L_0 > 1,\: k = 0, 1.
 \end{equation} 
  In fact, for $|x| \leq b,$ the sharp Huyghens  principle yields $g_0(t,x) = 0$ for $|t| > a + b$ and  (\ref{eq:3.9}) follows after an integration by parts.  The estimate (\ref{eq:3.9}) implies
\[ \|\pa_x^{\alpha} w_0\|_{k, b} \leq C_{a, b} |\lambda|^{k-1} \|\chi\|_{m, a}, \:\: |\alpha| \leq m,\: k = 0, 1\]
and  from the equation $\Delta \pa_x^{\alpha}w_0 - \lambda^2 \pa_x^{\alpha}w_0 = \pa_x^{\alpha}{\tilde{\chi}}$ we deduce 
\[ \|\pa_x^{\alpha} w_0 \|_{m-1, b} \leq C_{m, a, b} |\lambda|^{k-2} \|\chi\|_{m, a},\: |\alpha| = k = 0, 1.\]
 Since $m - 3 > d/2,$ the functions $w_0(\lambda, x)$ and $\pa_{x_j} w_0(\lambda, x), j =1,...,d$ are continuous in $|x| \leq b$ with uniform bound $\oc(|\lambda|^{-2})$ and $\oc(|\lambda|^{-1}),$ respectively, and
\[ \| q(\lambda, x)\|_{H^{m - 3/2}(\Gamma)}\leq D_m |\lambda|^{-1} \|\chi\|_{m, a}.\]

 Below up to end of this subsection we suppose that $\mu \leq - c_{m- 3/2}.$  Introduce the analytic function
 \[ h(\lambda, x) = \cc(\lambda)^{-1} q(\lambda, x) \in H^{m-1/2}(\Gamma).\]
 Applying the estimates (\ref{eq:3.6}), we deduce
 \[ \|h(\lambda, x)\|_{H^{m - 3/2}(\Gamma)} \leq B_{m-3/2} D_m(\alpha |\lambda \mu|)^{-1}\|\chi\|_{m, a}.\]
 Next, let $w_1(\lambda, x) = K(\lambda) h(\lambda, x) \in H^{m-2}(\Omega)$ be the solution of the Dirichlet problem (\ref{eq:2.1}) with $f = h(\lambda, x).$ 
 Clearly,
 \begin{equation} \label{eq:3.10}
 (\pa_{\nu} - \lambda \gamma(x)) w_1(\lambda, x) \vert_{\Gamma} = (N(\lambda) - \lambda \gamma(x))h(\lambda, x) = q(\lambda, x).
 \end{equation}
On the other hand, for  $w_1(\lambda, x)$ we have the estimates
 \begin{equation} \label{eq:3.11}
 \sum_{k + j\leq m-2} |\lambda|^j  \|w_1\| _{H^{k} (\Omega)} \leq K_m \|h\|_{H^{m - 3/2}(\g)}\leq M_m (\alpha|\lambda \mu|)^{-1}\|\chi\|_{m, a}.
 \end{equation}
 Consequently, the inclusion $H^{m-3}(\Omega) \subset C(\Omega)$ implies
 $$|w_1(\lambda, x)| \leq M_m |\lambda|^{-2}|\mu|^{-1} \|\chi\|_{m , a}$$
  uniformly with respect to $\mu \leq -c_{m-3/2}$ and $x \in \Omega.$
   For $(t, x) \in \R_t \times \bar{\Omega}$ and $c \geq c_{m- 3/2}$ define
 \[ g_1(t, x) = \frac{1}{2 \pi i} \int_{\mu = -c} e^{\lambda t} w_1(\lambda , x) d \lambda  = \frac{e^{-ct} }{2 \pi } \int_{-\infty}^{\infty} e^{\ii k t} w_1(- c + \ii k, x) dk.\]
 We may shift the contour of integration to the left, hence $g_1(\lambda, x)$ does not depend  of the choice of $c \geq c_{m-3/2}.$ 
  It is easy to see that $g_1(t, x) \equiv 0$ for $t \geq 0.$ 
  For fixed $t_0 \geq 0, x_0 \in \Omega$ one obtains 
\[ |g_1(t_0, x_0)| \leq \frac{e^{- c t_0}}{2 \pi} \int_{-\infty}^{\infty} |w_1(-c + \ii k, x_0)| dk
 \leq  \frac{A e^{-  c t_0}}{c}\int_{-\infty}^{\infty} \frac{1}{c_{m- 3/2}^2 + k^2} dk \leq \frac{Be^{- c t_0}}{c}\] 
 and taking $c \to \infty$, we conclude that $g_1(t_0, x_0) = 0.$  
 
According to (\ref{eq:3.11}), we get
\[\|w_1\|_{H^{m -2 -j}(\Omega)} = \oc_m(|\mu|^{-1}|\lambda|^{- j- 1}), \: j = 0,...,m-2.\]
In particular, since $m \geq 5,$ we have $\|\Delta w_1\|_{L^2} \leq \|w_1\|_{H^{m- 3}}= \oc_m(|\mu|^{-1} |\lambda|^{-2}).$
Therefore,
\[(2\pi i)\Delta g_1 = \int_{\mu = -c}e^{\lambda t} \Delta w_1(\lambda, x) d\lambda =  \lim_{R \to \infty} \int_{-c - iR}^{-c + iR} e^{\lambda t}\lambda^2 w_1(\lambda, x) d \lambda\]
\[ =  \lim_{R \to \infty} \pa_t^2 \int_{-c - i R}^{- c + iR} e^{\lambda t} w_1(\lambda, x) d \lambda = (2 \pi i )\pa_t^2 g_1(t, x).\]
Here in the last equality we take the derivative $\pa_t^2$ in the sens of distributions and we exchange $\pa_t^2$ and the limit $R \to \infty.$
This implies 
$$(\pa_t^2- \Delta_x) g_1 = 0, \: t < 0, x \in \Omega.$$
Finally, taking into account (\ref{eq:3.8}) and (\ref{eq:3.10}) and applying the inverse Laplace transformation,  for $ t < 0$ we deduce 
 \[(\pa_{\nu} - \gamma\pa_t) g_1(t, x)\vert_{\Gamma} = \frac{1}{2\pi i} \int_{\mu = -c} e^{\lambda t} q(\lambda, x) d\lambda \]
\[= \frac{1}{2 \pi i} \int_{\mu = -c} e^{\lambda t} \Bigl(\int_{-\infty}^0 e^{-\lambda s}(\pa_{\nu} - \gamma \pa_s)g_0(s, x)\vert_{\Gamma}ds\Bigr)d\lambda = (\pa_{\nu} -  \gamma \pa_t)g_0(t,  x)\vert_{\Gamma}.\]
Thus we obtain a solution $f(t,x) = g_0(t, x) - g_1(t, x)$ of the problem (\ref{eq:1.3}).
 Introduce the space $\mathcal H^m(\R_t^{-} \times \Omega)$ with norm
\[\|| w(t, x) \||_m = \sum_{j + |\alpha| \leq m} \|\pa_t^j \pa_x^{\alpha} w(t, x) \|_{L^2(\R_t^{-} \times \Omega)} .\]
The estimate (\ref{eq:3.11})  implies $e^{c_{m- 3/2} t} f(t, x) \in \mathcal H^{m-2}(\R_t^{-} \times \Omega)$ and we obtain the following
\begin{prop}
For $m >d/2 + 3$ there exists an operator
\begin{equation} \label{eq:3.12}
P(t) :H^m_{comp} (\Omega) \longrightarrow e^{-c_{m- 3/2} t} \mathcal H^{m-2}( \R_t^{-} \times \Omega)
\end{equation} 
satisfying
\begin{equation} \label{eq:3.13}
\begin{cases}
 (\pa_t^2 - \Delta_x) P(t) = 0, \: P(0) = 0, \: \pa_tP (0) = \id,\\
(\pa_{\nu} - \gamma(x) \pa_t) P(t) \vert_{R_t^{-} \times \Gamma} = 0.\end{cases}
\end{equation}
\end{prop}

\section{Absence of eigenvalues $|\lambda| \geq A_0$}

To obtain  an  analytic continuation of the incoming resolvent for $ |\lambda| \geq A_0,$ we will follow the approach in section 4.6  in \cite{DZ}. Let $\chi_a \in C^{\infty}_0(\R^d)$ be  such that $\supp \chi_a \subset B(0, a)$ and $\chi_a \equiv1$ for $|x| \leq \rho + 2\ep, \: a \geq \rho +3 \ep, \:\ep > 0.$
Introduce a function $\zeta_a(t, x) \in C^{\infty}$ such that
$$\zeta_a(t, x) \vert_{|x| > \rho + 2 \ep} = \begin{cases} 1 \:\: {\rm for}\: 0 \geq t \geq -(|x| + T_a),\\0 \:\:{\rm for}\: t \leq -(x| + T_a + 1), \end{cases}$$
and $\zeta_a(t, x) = \varphi_a(t) $ for $|x| < \rho + \ep$ with 
$$\varphi_a(t) = \begin{cases} 1 \:\:{\rm for}\: 0 \geq t  \geq -(\rho + \ep + T_a),\\
0 \:\: {\rm for}\: t  \leq -(\rho + \ep + T_a + 1). \end{cases}$$
Here $T_a \geq a$ is a fixed constant. Moreover, we suppose that $\zeta_a(t, x) = 1$ for $0 \geq t \geq -T_a, \: x \in \Omega.$ Then for $|x| \leq \rho + \ep$ we have $\pa_t \zeta_a(t, x) = \psi_a(t) \in C_0^{\infty}(\R_t^{-}).$ 

 For $m > d/2 + 3$ consider the operator $P(t)$ given by Proposition 3.2 and introduce
$$\zeta_a P(t) \chi_a:  H^m(\Omega) \longrightarrow e^{-c_{m- 3/2} t} \mathcal H^{m-2}(\R_t^{-} \times \Omega).$$
 Therefore,
\[(\pa_t^2 - \Delta_x) (\zeta_a P(t) \chi_a)  = [\square , \zeta_a] P(t) \chi_a = : F_a(t)\]
and 
\[( \pa_{\nu} - \gamma(x) \pa_t) (\zeta_a P(t) \chi_a)\vert_{\R_t^{-} \times \g} =  -\gamma(x) \psi_a(t) P(t) \chi_a\big\vert_{\R_t^{-} \times \g} = : J_a(t).\]
Clearly, for $g \in H^m(\Omega)$ we have
$$\supp (F_a(t)g) \subset \{(t, x) \in \R_t^{-} \times \Omega:\: -(\max\{|x|, \rho + \ep\} + T_a + 1) \leq t \leq - T_a)\},$$
$$\supp (J_a(t)g) \subset \{ (t, x) \in \R_t^{-} \times \Gamma:\: -(\rho + \ep + T_a + 1) \leq t \leq - (\rho + \ep + T_a)\}.$$
The operator $P(t)$ has a Schwartz kernel $p(t, x, y) \in \mathcal D'( \R_t^{-} \times \Omega \times \R^d_y)$ such that 
\[ (P(t) f)(x) = \int_{\R^d} p(t, x, y) f(y) dy,\: f \in C_0^{\infty} (\R^d),\]
where the integral is taken in the sens of distributions. This implies that $J_a(t)$ has kernel $j_a(t, x, y) = \gamma(x) \psi_a(t) p(t, x, y) \chi_a(y)\vert_{x \in \g}.$
According to (\ref{eq:3.13}), we have 
$$\gamma(x)\psi_a(t) P(t)\chi_a \vert_{\Gamma} : H^{m}(\Omega) \rightarrow \mathcal H^{m -5/2}(\R_t^{-} \times \Gamma).$$ 
We will construct a kernel $\tilde{b}_a(t, x, y)$ such that
\begin{equation} \label{eq:4.1}
 \tilde{b}_a(t, x, y) = 0 \: {\rm for}\: -(\rho + \ep + T_a + 1) \leq t \leq 0, \: x \in \Gamma,
\end{equation}
\begin{equation}\label{eq:4.2}
 \pa_{\nu}\tilde{b}_a(t, x, y)= j_a(t, x, y) \: {\rm for}\: -(\rho + \ep + T_a + 1) \leq t \leq 0,\: x \in \Gamma.
\end{equation}
 
To arrange the above conditions we apply Remark 2.1 considering the variables $t, y$  as parameters.  By the construction with local coordinates, we have $\tilde{b}_a(t, x, y) = \alpha(x)j_a(t, x, y)$ with some smooth function satisfying $\alpha(x)\vert_{\g} = 0, \: \pa_{\nu} \alpha(x) \vert_{\g} = 1.$ Choose $\chi_b \in C_0^{\infty}(B(0, a))$ such that $\chi_b = 1$ near $B(0, \rho+ \ep)$ and $\chi_a \equiv1$ on the support of $\chi_b$ and define $b_a(t, x, y) = \chi_b(x) \tilde{b}_a(t, x, y)$. The conditions (\ref{eq:4.1}), (\ref{eq:4.2}) yield
\[(\pa_{\nu} - \gamma(x) \pa_t)b_a(t, x, y) \vert_{x \in \Gamma} = j_a(t, x, y).\]Next let $B_a(t): H^m(\Omega) \longrightarrow H^{m- 2}(\R_t^{-} \times \Omega)$ be the operator with kernel $b_a(t, x, y).$ The smoothness of $b_a(t, x, y)$ follows form that of 
$j_a(t, x, y)$. The choice of the correction $B_a(t)$ leads to the equality
\[ (\pa_{\nu} - \gamma \pa_t)\Bigl (\zeta_a P(t) \chi_a - B_a(t)\Bigr)\Big\vert_{x \in \Gamma} = 0\]
and we have
\[ \supp (B_a(t)g) \subset \{ (t, x): \: -(\rho +\ep + T_a + 1) \leq t \leq -(\rho + \ep + T_a),\: |x| \leq a\}.\]

Denote by $W_a(t)$ the solution of the problem
\begin{equation} \label{eq:4.4}
\square_0 W_a(t) =   -(1- \chi_b)F_a(t), 
 W_a(t) \equiv 0 \: {\rm for}\: t \geq 0.\
\end{equation} 
Here $\square_0 = \pa_t^2 - \Delta_0$ and $\Delta_0$ is the Laplacian in $\R^d.$
Repeating the argument of the proof of Theorem 4.3 in \cite{DZ}, it easy to see that there exists $C_a >T_a$ such that
\begin{equation}\label{eq:4.5}
{\rm supp} \: (W_a(t) g) \subset \{(t, x) \in \R_t^{-} \times \Omega: \:-( |x| + C_a) \leq t \leq -(|x| + T_a)\}.
\end{equation}
For completeness we present the details.  Write
    $$W_a(t) = -(1 - \chi_b) \zeta_a H(t)P(t) \chi_a  +H(t)U_0(t)((1 - \chi_b) \chi_a) +  Q_a(t),$$
   where $U_0(t)$ is the solution of the Cauchy problem
    \[ \begin{cases} \square_0 U_0 = 0,\\
    U_0(0) = 0,\: (\pa_t U_0) (0) = ( 1- \chi_b)\zeta_a \chi_a.\end{cases} .\]
    Then 
    $$\square_0 Q_a = \square_0 W_a +  (1- \chi_b) \square_0 \Bigl(\zeta_aH(t)P(t)\chi_a)\Bigr)  + [\Delta, \chi_b] \zeta_a H(t) P(t) \chi_a$$
    $$- \square_0 (H(t) U_0(t) (1 - \chi_b)\chi_a))= (1 - \chi_b) \chi_a + [\Delta, \chi_b] \zeta_a H(t) P(t) \chi_a  - (1-\chi_b) \chi_a $$
    $$=  [\Delta, \chi_b] \zeta_a H(t) P(t) \chi_a = r_{a, b}(t, x) ,$$
    $$Q_a(0) = (\pa_t Q_a)(0) = 0.$$
    We apply the Duhamel formula
    $$Q_a(t) = -\int_{t}^0 U_0(t- s, x ,y) r_{a,b}(s, y)dyds,$$
    where 
    $U_0(t, x, y)$ is the kernel of the operator $\frac{\sin(t \sqrt{-\Delta_0})}{\sqrt{-\Delta_0}}.$ By the sharp Huyghens principle with a constant $C_a > 0$ we obtain
    $$\supp \Bigl(Q_a(t)g\Bigl) \subset \{(t, x): \: -(|x| + C_a) \leq t \leq 0\}.$$
    Indeed, if $(t, x) \in \supp Q_a(t)g$, there exist $( s, y) \in \supp r_{a, b}$ and $t \leq s \leq 0$ such that $|x-y| = -(t- s)$ and this implies 
    $$- t \leq |x| + |y| - s \leq |x| + 2a + T_a + 1.$$
    For  $U_0(t)(1 - \chi_b) \chi_a$ by a finite speed of propagation argument one obtains the same result. This proves the lower bound of $t$ in (\ref{eq:4.5}). To establish the upper bound, we apply ones more the Duhamel formula for $W_a(t)$ writing
    \[ W_a(t)g = \int_t^0 U_0(t-s, x, y)  (1 - \chi_b(y)) F_a(s) g)(y) dy ds, \: t \leq 0.\]
    If $(t, x) \in \supp W_a(t) g,$  there exist $(s, y) \in \supp (1 - \chi_b) F_a(s)g(y)$ and $t\leq s \leq 0$ such that $|x- y| = -(t- s).$ Hence
    $$ -t \geq |x| - |y| - s\geq |x| - |y| + |y| + T_a = |x| + T_a.$$

 Choose a function  $\chi_c \in C_0^{\infty}(B(0, a))$ which is equal to 1 near $B(0, \rho)$ and such that $\chi_b \equiv
    1$ on $\supp \chi_c$ and introduce the approximative resolvent
    $$R^{\sharp}_a(\lambda) = \int_{-\infty}^0 e^{-\lambda t} \Bigl[\zeta_a P(t)\chi_a  + (1- \chi_c) W_a(t)- B_a(t) \Bigr]dt, \: \Re \lambda <  0.$$
    The integral is well defined since for fixed $x$ the supports with respect to $t$ of the functions under integration are bounded. For $m > d/2 + 3$ we have
    $R^{\sharp}_a(\lambda): \: H_{comp}^m(\Omega) \rightarrow H^{m-2}_{loc} (\Omega).$
    We obtain
    \[ (\Delta - \lambda^2) R^{\sharp}_a(\lambda) =  \chi_a \Bigl( \id+ L_a(\lambda)\Bigr), \:  \ (\pa_{\nu} - \lambda \gamma(x) ) R^{\sharp}_a(\lambda)\vert_{x \in\Gamma} = 0\]
     with  
    \[ L_a(\lambda) = \int_{-\infty}^0 e^{-\lambda t} \Bigl ( -\chi_b F_a(t) -  [\Delta, \chi_c] W_a(t)+ \square B_a(t)\Bigr) dt.\]
 
    After an integration by parts with respect to $t$ we deduce
    \[ L_a(\lambda) = \lambda^{-1} \int_{-\infty}^0 e^{-\lambda t} \Bigl( -\chi_b \pa_t F_a(t) - [\Delta, \chi_c] \pa_tW_a(t) + \pa_t \square B_a(t)\Bigr) dt. \]
    For the invertibility of $\id + L_a(\lambda)$ in $H^m(\Omega \cap B(0, a))$ we must estimate the $\|.\|_{m, a}$ norms of the terms under  integration. For $\pa_t F_a(t)$ one observes that  $\pa_t^2  P(t) = \Delta P(t)$, hence  $\|\pa_t F_a(t)\|_{m, a} \leq C_{1, m}\|\chi_a\|_{m + 4}, a.$ Similarly, for $\square B_a(t)$ we exploit once more the property of $P(t)$ since $b_a(t, x, y) = \chi_b(x)\alpha(x) p(t, x, y)$ and  conclude that $\|\pa_t \square B_a(t)\|_{m, a} \leq C_{2, m} \|\chi_a\|_{m + 4, a}$, Finally, $\pa_t W_a(t)$ can be estimated by $F_a(t)$ and we deduce 
    \[ \| L_a(\lambda)\|_{m, a} \leq D_{m+ 4, a} |\lambda|^{-1} \|\chi_a\|_{m+ 4, a}.\]
 Let $\rho < a_1 < a$. We fix $m > d/2 + 3$ and fix a function $\chi_a \in C_0^{\infty}(B(0, a))$ such that $\chi_a \equiv 1$ for $|x| \leq a_1.$ Choose $|\lambda| \geq 2D_{m + 4, a} \|\chi_a\|_{m + 4, a} = A_0.$   Then the operator $( \id + L_a(\lambda))^{-1}$ is invertible in $H^m(\Omega \cap B(0, a))$ and  for every $g \in H^m(\Omega)$ with $\supp g \subset B(0, a_1)$ by an analytic continuation we get
\[R_a(\lambda) g = R_a^{\sharp}(\lambda) (\id + L_a(\lambda))^{-1} g = \chi_a g = g,\: \Re \lambda < 0, \:|\lambda| \geq A_0.\] 
The function $u(x,\lambda) = R_a(\lambda)g$ is  analytic solution of (\ref{eq:2.4}) with $f = 0$. 
 According to Lemma 2.1, this implies that for every $s \geq 0$ the operator $\cc(\lambda)^{-1}: H^s(\g) \to H^{s+1}(\g)$ is analytic for $\Re \lambda < 0,\: |\lambda| \geq A_0$
 and we obtain the following
 \begin{prop} There exists $A_0 > 0$ such that the operator $G$ has no eigenvalues 
 $$\lambda \in \{\C: \: \Re \lambda \leq - c_0\} \cup \{\C: \: \lambda \in \C_{-},\: |\lambda| \geq A_0\}.$$
\end{prop}
 It is important to note that $D_{m + 4, a}$ depends of constants $A_{m+4}, B_{m+ 4}, K_{m+4}$ in subsection 3.2.
 
\section{Absence of eigenvalues  $|\lambda| \leq A_0$}
Consider for $\ep \in [0, 1]$ the operator
$$\cc_{\ep}(\lambda) : = N(\lambda) - \lambda \ep \gamma(x) = \Bigl( \id - \lambda \ep\gamma(x)N(\lambda)^{-1}\Bigr)N(\lambda).$$
Let $e^{t G_{\ep}}$ be the semi-group with generator $G_{\ep}$ defined with boundary condition $(\pa_{\nu} - \ep \gamma(x) \pa_t) f\vert_{\g} = 0.$
 Obviously, $\max_{x \in \g} \ep \gamma(x) = \ep m_0 < 1.$ Set $\alpha_{\ep} = 1 - \ep m_0 \geq \alpha.$ 
 Below for simplicity we write $H^s$ instead of $H^s(\Gamma).$ It is easy to see that for $\mu  \leq -c_0$ we have
\begin{equation}\label{eq:5.2}
 - \Re ( (N(\lambda) - \lambda \ep \gamma(x)) \varphi, \varphi)_0 \geq (\alpha_{\ep} |\mu| - C_0) \|\varphi\|_0^2,\:\: \forall \varphi \in C^{\infty}(\g)
 \end{equation}
with constant $C_0 > 0$ independent of $\ep$ and $\mu.$ By the argument in subsection 3.1 we conclude that the operator $\cc_{\ep}(\lambda)^{-1}: H^s \to H^{s+1}$ is analytic for $\mu \leq -c_0$ uniformly with respect to $\ep.$  Recall the constants $c_m = \max \big\{ \frac{2 C_m}{\alpha}, c_{m-1}\big\}> 0$ defined in subsection 3.1. Since  $C_m$ in (\ref{eq:3.2}) are independent of $\gamma(x),$ we obtain 
\[c_{m, \ep} = \max\Big\{\frac{2 C_m}{\alpha_{\ep}}, c_{m-1,\ep}\Big\} \leq c_m.\]
Moreover, for $\mu \leq -c_m$ we have  estimates
\begin{equation}  \label{eq:5.2}
\|\cc_{\ep}(\lambda)^{-1} h\|_s \leq \frac{B_s}{\alpha_{\ep} |\mu|} \|h\|_s \leq \frac{B_s}{\alpha |\mu|},\: \forall h \in H^s(\g)
\end{equation}
with  $B_s > 0$ independent of $\ep.$ To do this, we repeat the argument of subsection 3.1 and recall  that the constants $A_m$ depend of the norms of the derivatives of $\ep \gamma(x)$ of order $j \leq m$, hence they can be chosen to be independent of $\ep.$  Repeating the proof of subsection 3.2, we construct an operator 
\[P_{\ep}(t) :H^m_{comp} (\Omega) \longrightarrow e^{-c_{m- 3/2} t} \mathcal H^{m-2}( \R_t^{-} \times \Omega)\]
satisfying
\begin{equation} \label{eq:5.3}
\begin{cases}
 (\pa_t^2 - \Delta_x) P_{\ep}(t) = 0, \: P_{\ep}(0) = 0, \: \pa_t P_{\ep} (0) = \id,\\
(\pa_{\nu} - \ep\gamma(x) \pa_t) P_{\ep}(t) \vert_{R_t^{-} \times \Gamma} = 0.\end{cases}
\end{equation}
The analysis of Section 4 implies that $\cc_{\ep}(\lambda)^{-1}$ has an analytic continuation for $\Re \lambda < 0, \: |\lambda| \geq A_0,$ where $A_0 > 0$ is independent of $\ep.$ This follow from the fact that the constants $A_{m+4}, B_{m+ 4}, K_{m+4}$ in subsection 3.2 can be chosen uniformly with respect to $\ep.$ 
 Denote by $A(\lambda, \ep)$ the compact operator
$$A(\lambda, \ep) =  \lambda \ep\gamma(x)N(\lambda)^{-1} : \: H^{1/2}\rightarrow H^{3/2}.$$
 Consider the  domain $\omega \subset \bar{C}_{-}$ bounded by the the contour 
$$\zeta = \{z \in \C: \Im z = -A_0, \: - c_0 \leq \Re z \leq 0\} \cup \{z \in \C: \Re z = 0,\: |\Im z| \leq A_0 \} $$
$$\cup \{z \in \C: \Im z = A_0, \: - c_0 \leq \Re z \leq 0\} \cup \{z \in \C: \Re z = - c_0,\: |\Im z| \leq A_0\}.$$

 For sufficiently small $\ep$ the operator $C _{\ep}(\lambda)$ is invertible for $\lambda \in \omega.$ Indeed, for all $\lambda \in \omega$ we have $|\lambda| \leq B_0$ and $\|\gamma(x)N(\lambda)^{-1}\|_{H^{1/2} \to H^{1/2}} \leq B_1.$
 Then for $\ep \leq \frac{1}{2 B_0B_1}= \eta_0 $ one arranges
\[\| \lambda \ep \gamma(x) N(\lambda)^{-1} \|_{H^{1/2} \to H^{1/2}} \leq 1/2,\: \lambda \in \omega.\]
This implies that $(\id - A(\lambda, \ep))$ is invertible in $H^{1/2}$ and $\cc_{\ep}(\lambda)^{-1}$ is analytic for $\lambda \in \omega.$ Define
\[ \eta = \sup\{\ep:\: \ep \in [0, 1]:\: \cc_{\ep}(\lambda)^{-1} \: {\text {is analytic for}\: \lambda \in \omega} \}.\]
Obviously, $\eta_0 \leq \eta \leq 1.$
Notice that if $\cc_{\ep}(\lambda)$ and $\cc_{\eta}(\lambda)$ are invertible, we have the equality
\begin{equation} \label{eq:5.4}
\cc_{\ep} (\lambda)^{-1}  \Bigl( \id -(\ep - \eta) \lambda \gamma \cc_{\eta}(\lambda)^{-1}\Bigr) = \cc_{\eta}(\lambda)^{-1}. 
\end{equation}

We have two cases: (i) $\eta = 1,$ (ii) $\eta < 1.$ If $\eta = 1$ and $\cc(\lambda) = \cc_1(\lambda)$ is invertible for $\lambda \in \omega,$ the operator $G$ has no eigenvalues in $\omega$.  In the case (ii)  if  $\cc_{\eta}(\lambda)$ is invertible for $\lambda \in \omega$, by using (\ref{eq:5.4}), we obtain that $\cc_{\ep}(\lambda)$ is invertible for $1 > \ep > \eta, \: \lambda \in \omega$
 and $\ep - \eta$ sufficiently close to 0. To prove this, for $\ep - \eta > 0$ sufficiently small and $\cc_{\ep}(\lambda)^{-1}$ analytic, we write
 \begin{equation} \label{eq:5.5}
 \cc_{\ep}(\lambda)^{-1} = \cc_{\eta}(\lambda)^{-1} \Bigl( \id -(\ep - \eta) \lambda \gamma \cc_{\eta}(\lambda)^{-1}\Bigr)^{-1}.
 \end{equation}
 The right hand part remains bounded for $\lambda \in \omega$ and $\ep - \eta > 0$ small and we deduce the analyticity of $\cc_{\ep}(\lambda)^{-1}$ for $\lambda \in \omega$ by analytic continuation.
  This implies a contradiction with the  definition of $\eta.$ Consequently, in both cases (i) and (ii) we can assume that $\cc_{\eta}(\lambda)^{-1}$ has a pole $\lambda_0 \in \mathring{\omega},$ where $\mathring{\omega}$ denotes the interior of $\omega.$
  
 Choose a small $\delta > 0$ so that $\{z \in \C:\: |z- \lambda_0|\leq \delta \} \subset \mathring{\omega}$ and $\cc_{\eta}(\lambda)^{-1}$ has no poles $\lambda$ with $0 < |\lambda - \lambda_0| \leq \delta.$ Let  $\gamma_0 = \{ z  \in \C: \: |z- \lambda_0| = \delta\}$ be positively oriented.  Consider the projector
 \[\Pi_{\eta}(\lambda_0)  = \frac{1}{2 \pi \ii} \int_{\gamma_0} (\lambda - G_{\eta})^{-1} d\lambda \]
having rank equal to the multiplicity of the eigenvalue $\lambda_0$ of $G_{\eta}.$ With a constant $D_0 > 0$ independent of $\lambda$ one has  a bound $\|(G_{\eta} - \lambda)^{-1}\|_{\mathcal H} \leq D_0,\: \lambda \in \gamma_0.$ We claim that for $0 <\ep < \eta$ and $\eta - \ep$ sufficiently small we have
\begin{equation} \label{eq:5.6}
 \|(G_{\eta} - \lambda)^{-1} - (G_{\ep}- \lambda)^{-1} \|_{\mathcal H} \leq C_0 |\ep - \eta|,\: \forall \lambda \in \gamma_0.
\end{equation}
with $C_0 > 0$ independent of $\ep$ and $\lambda.$
This follows from a perturbation argument (see Section 5 in \cite{CPR}).  We present a direct proof. 
First, for $\lambda \in \gamma_0$ the operator $\cc_{\eta}(\lambda)$ is invertible and $\|\gamma(x) \cc_{\eta}(\lambda)^{-1}\|_{H^{1/2} \to H^{3/2}}\leq D_1, \: \lambda \in \gamma_0$. Applying (\ref{eq:5.5}) for $|\ep - \eta| \leq \frac{1}{2 B_0 D_1}$ and $\lambda \in \gamma_0$ we deduce a bound 
\[ \|\cc_{\ep}(\lambda)^{-1}\|_{H^{1/2} \to H^{1/3}}\]
\[ \leq \|\cc_{\eta}(\lambda)^{-1}\|_{H^{1/2} \to H^{3/2}} \Big\|\Bigl(\id - (\ep - \eta) \lambda \gamma \cc_{\eta}(\lambda)^{-1}\Bigr)^{-1}\Big\|_{H^{1/2} \to H^{1/2}} \leq 2 D_1.\]
Taking into account the bounds for $\cc_{\eta}(\lambda)^{-1}$ and $\cc_{\ep}(\lambda)^{-1},$ from (\ref{eq:5.4}) we get
\[\| \cc_{\eta}(\lambda)^{-1} - \cc_{\ep}(\lambda)^{-1}\|_{H^{1/2} \to H^{3/2}} \leq 2 |\ep - \eta| B_0 D_1^2,\: \lambda \in \gamma_0.\]
Next we exploit the representation (\ref{eq:2.3})  for the first components $u_{\eta}(x, \lambda)$ and $u_{\ep}(x, \lambda)$ of $(G_{\eta} - \lambda)^{-1}(f,g)$ and $(G_{\ep} - \lambda)^{-1}(f, g)$ with $\cc_{\eta}(\lambda)^{-1}$ and $\cc_{\ep}(\lambda)^{-1},$ respectively. Choose a function $\varphi \in C_0^{\infty}(\R^d)$  equal to 1 for $|x| \leq a,\: a > \rho.$ For $\lambda \in \gamma_0$ we obtain
\[\big \|( \cc_{\eta}(\lambda)^{-1} - \cc_{\ep}(\lambda)^{-1}) \Bigl(\pa_{\nu} R_D(\lambda)(g + \lambda f)\vert_{\g} + \gamma f\vert_{\g}\Bigr)\big\|_{H^{3/2}} \]
\[\leq |\ep - \eta| B_0 D_2 D_1^2 \Bigl(\|\varphi R_D(\lambda)(g + \lambda f) \|_{H^1(\Omega)}+ \|\varphi f\|_{H^1(\Omega)}\Bigr).\]
We prove the claim  (\ref{eq:5.6}) by using the estimates for the resolvent $R_D(\lambda)$ and the operator $K(\lambda).$ Next for $\eta - \ep > 0$ sufficiently small we obtain
\begin{equation} \label{eq:5.7}
\Big\| \frac{1}{2\pi \ii}\int_{\gamma_0} \Bigl((\lambda - G_{\eta})^{-1} - (\lambda - G_{\ep})^{-1}\Bigl) d\lambda\Big \|_{\mathcal H} \leq C_0 \delta |\ep-\eta|.
\end{equation}
The analyticity of $\cc_{\ep}(\lambda)^{-1}$ for $\ep < \eta$ implies $\int_{\gamma_0} (\lambda - G_{\ep})^{-1} d\lambda = 0$. For  small $\eta- \ep$ the estimate (\ref{eq:5.7}) leads to $\|\Pi_{\eta}(\lambda_0)\|_{{\mathcal H} \to {\mathcal H}} < 1$ and this yields a contradiction with the existence of the pole $\lambda_0.$ This completes the proof of Theorem 1.1.


\begin{thebibliography}{99} 
 
  \bibitem{CPR} F. Colombini, V. Petkov and J. Rauch, {\em Spectral problems for non elliptic symmetric systems with dissipative boundary conditions}, Journal of Funct. Analysis, {\bf 267} (2014), 1637-1661. 
 \bibitem{DZ} S. Dyatlov and M. Zworski, Mathematical Theory of Scattering Resonances. Graduate
Studies in Mathematics. American Mathematical Society, 2019.
 
 \bibitem{I1}M. Ikawa, {\em Problèmes mixtes pour l'équation des ondes}, Publ. RIMS, Kyoto Univ. {\bf 12} (1976), 55-121.
 
  \bibitem{I2} M. Ikawa, {\em Problèmes mixtes pour l'équation des ondes, ${\rm II}$}, Publ. RIMS, Kyoto Univ. {\bf 13} (1977), 61-106.
 
 \bibitem{I3} M. Ikawa, {\em Mixed problems for the wave equation ${\rm IV}$, The existence and the  exponential decay of solutions}, J. Math. Kyoto Univ. {\bf 19} (3) (1979), 375-411. 
 
 
  \bibitem{LP1} P. Lax and R. Philips, {\em Scattering theory for dissipative hyperbolic systems}, J. Funct. Anal. {\bf 14} (1973), 172-235.

\bibitem{Ma} A. Majda, {\em The location of the spectrum of the dissipative acoustic operator}, Indiana Univ. Math. J. {\bf 25} (10) (1976), 973-987.
  
 \bibitem{P1} V. Petkov,  {\em Location of the eigenvalues of the wave equation with dissipative boundary conditions}, Inverse Problems and Imaging, {\bf 10} (4) (2016), 1111-1139.
  
  \bibitem{P2} V.Petkov, {\em Weyl formula for the eigenvalues of the dissipative acoustic operator}, Res. Math. Sci.  {\bf 9} (1) (2022), Paper 5.  
  
 \bibitem{P3}  V. Petkov, {\em Eigenvalues and resonances of dissipative acoustic operator for strictly convex obstacles}, arXiv: 2310.01192. math.AP. 

  \end{thebibliography}
\end{document}